\documentclass[12pt]{article}
\usepackage{graphics}
\usepackage{amsmath}
\usepackage{amsfonts}
\usepackage{enumerate}
\usepackage{latexsym}
\usepackage{epsfig}

\newtheorem{theorem}{Theorem}[section]
\newtheorem{lemma}[theorem]{Lemma}

\def\proofbox{\begin{picture}(6.5,6.5)
\put(0,0){\framebox(6.5,6.5){}}\end{picture}}
\newenvironment{proof}{\noindent{\it Proof.\quad}}{\hfill\proofbox}

\begin{document}

\title{Simplicial Maps of the Complexes of Curves on Nonorientable Surfaces}
\author{Elmas Irmak}

\maketitle

\renewcommand{\sectionmark}[1]{\markright{\thesection. #1}}


\thispagestyle{empty}
\maketitle



\begin{abstract} Let $N$ be a compact, connected, nonorientable surface of genus $g$ with $n$ boundary components.
Let $\lambda$ be a simplicial map of the complex of curves, $\mathcal{C}(N)$, on $N$ which satisfies the following:
$[a]$ and $[b]$ are connected by an edge in $\mathcal{C}(N)$ if and only if $\lambda([a])$ and $\lambda([b])$ are
connected by an edge in $\mathcal{C}(N)$ for every pair of vertices $[a], [b]$ in $\mathcal{C}(N)$. We prove that
$\lambda$ is induced by a homeomorphism of $N$ if $(g, n) \in \{(1, 0), (1, 1), (2, 0)$, $(2, 1), (3, 0)\}$ or
$g + n \geq 5$. Our result implies that superinjective simplicial maps and automorphisms of $\mathcal{C}(N)$
are induced by homeomorphisms of $N$. \end{abstract}

\maketitle

{\small Key words: Mapping class groups, simplicial maps, nonorientable surfaces

}

\section{Introduction}

Let $N$ be a compact, connected, nonorientable surface, and $\mathcal{C}(N)$ be the complex of curves of $N$. The author proved that superinjective simplicial of maps of $\mathcal{C}(N)$ are induced by homeomorphisms of $N$ in \cite{Ir5}. The referee of that paper suggested that the following stronger result should hold. In this paper we give a proof of the referee's statement. Our main result in this paper implies that superinjective simplicial maps and automorphisms of the complexes of curves on nonorientable surfaces are induced by homeomorphisms of these surfaces.
The main result is the following:

\begin{theorem} Let $N$ be a compact, connected, nonorientable surface of genus $g$ with
$n$ boundary components. Suppose that either $(g, n) \in \{(1, 0), (1, 1), (2, 0)$, $(2, 1), (3, 0)\}$ or $g + n \geq 5$.
If $\lambda : \mathcal{C}(N) \rightarrow \mathcal{C}(N)$ is a simplicial map satisfying: $[a]$ and $[b]$ are connected by an edge in
$\mathcal{C}(N)$ if and only if $\lambda([a])$ and $\lambda([b])$ are connected by an edge in $\mathcal{C}(N)$ for every pair of
vertices $[a], [b]$ in $\mathcal{C}(N)$, then $\lambda$ is
induced by a homeomorphism $h : N \rightarrow N$ (i.e $\lambda([a]) = [h(a)]$ for every vertex $[a]$ in $\mathcal{C}(N)$).\end{theorem}

We note that $g + n = 4$ case is open.

The extended mapping class groups are defined as the group of isotopy classes of self-homeomorphisms of the surface on orientable surfaces. Ivanov proved that the automorphism group of the curve complex is isomorphic to the extended mapping class group on orientable surfaces.
As an application he proved that isomorphisms between any two finite index subgroups are geometric. Ivanov's results were proven by Korkmaz
in \cite{K1} for small genus cases, and by Luo for all cases in \cite{L}. Classification of injective homomorphisms between mapping class groups was given by Ivanov-McCarthy \cite{IMc}. To see some results where the mapping class group was viewed as the automorphism group of some other complexes on orientable surfaces, see of results of  Schaller \cite{S}, Margalit \cite{M}, the author \cite{Ir3}, Brendle-Margalit \cite{BM1}, McCarthy-Vautaw\cite{MV}, Farb-Ivanov \cite{FIv}, Irmak-Korkmaz \cite{IrK}, Irmak-McCarthy \cite{IrM}.

Superinjective simplicial maps of complex of curves were defined on the orientable surfaces by the author in \cite{Ir1} as maps that preserve geometric intersection zero and nonzero properties of the vertices in the complex of curves. After proving that superinjective simplicial maps of the complexes of curves are induced by homeomorphisms of the surface, the author classified injective homomorphisms from finite index subgroups of the extended mapping class group to the whole group in \cite{Ir1}, \cite{Ir2}, \cite{Ir3}. Behrstock-Margalit and Bell-Margalit proved the author's results for small genus cases in \cite{BhM} and in \cite{BeM}. There are several results about superinjective simplicial maps and the applications to the mapping class groups given by Brendle-Margalit in \cite{BM1}, Kida in \cite{Ki1}, \cite{Ki2}, \cite{Ki3}, Kida-Yamagata in \cite{KiY1}, \cite{KiY2}, \cite{KiY3}. Shackleton proved that injective simplicial maps of the curve complex are induced by homeomorphisms in \cite{Sc}, and he obtained applications for mapping class groups.

On nonorientable surfaces the complex of curves also was viewed as the automorphism group of some complexes on surfaces. To see this, see the results of Atalan \cite{A}, the author \cite{Ir4}, Atalan-Korkmaz \cite{AK}.

\section{Simplicial maps that satisfy the connectivity property}

We will always assume that $N$ is a compact, connected, nonorientable surface of genus $g$ with $n$ boundary components.
The complex of curves, $\mathcal{C}(N)$, on $N$ is an abstract simplicial complex. The vertex set is the set of isotopy
classes of nontrivial simple closed curves. Nontrivial means that the curve does not bound a disk, a mobius band, and
it is not isotopic to a boundary component of $N$. A set of vertices forms a simplex if they can be represented by pairwise
disjoint simple closed curves on $N$.

We will say that a simplicial map $\lambda : \mathcal{C}(N) \rightarrow \mathcal{C}(N)$ satisfies the connectivity property
if the map satisfies the following: $[a]$ and $[b]$ are connected by an edge in $\mathcal{C}(N)$ if and only if $\lambda([a])$
and $\lambda([b])$ are connected by an edge in $\mathcal{C}(N)$ for every pair of vertices $[a], [b]$ in $\mathcal{C}(N)$.

We note that a simple closed curve is called 2-sided if its regular neighborhood is an annulus. It is called 1-sided if its regular
neighborhood is a Mobius band. We remind that for any two vertices $[a]$ and $[b]$ in $\mathcal{C}(N)$ the geometric intersection
number of $[a]$, $[b]$, $i([a], [b])$, is defined as the minimum number of points of $x \cap y$ where $x \in [a]$ and $y \in [b]$.

A set $P$ of pairwise disjoint, nonisotopic, nontrivial simple closed curves on $N$ is called a pair of pants decomposition
of $N$ if each component of the cut surface $N_P$ is a pair of pants.
Let $a$ and $b$ be two distinct elements in a pair of pants decomposition $P$. Then $a$ is called adjacent to $b$ w.r.t. $P$
iff there exists a pair of pants in $P$ which has $a$ and $b$ on its boundary. The set of isotopy classes of elements of $P$,
$[P]$, gives a maximal simplex of $\mathcal{C}(N)$. The vertex set of every maximal simplex of $\mathcal{C}(N)$ is equal to
$[Q]$ for some pair of pants decomposition $Q$ of $N$.

On orientable surfaces all maximal simplices have the same dimension $3g_1 + n_1 -4$ where $g_1$ is the genus and
$n_1$ is the number of boundary components of the orientable surface. On nonorientable surfaces there are different
dimensional maximal simplices. In Figure \ref{Fig0} we show some pair of pants decompositions on a closed, nonorientable 
surface of genus seven. They correspond to different dimensional maximal simplices in the complex of curves. The cross signs 
in the figure means that the interiors of the disks which have cross signs in them are removed and the antipodal points of 
the resulting boundary components are identified.

The following lemma is given in \cite{A}, \cite{AK}:

\begin{figure}
\begin{center}
\epsfxsize=2in \epsfbox{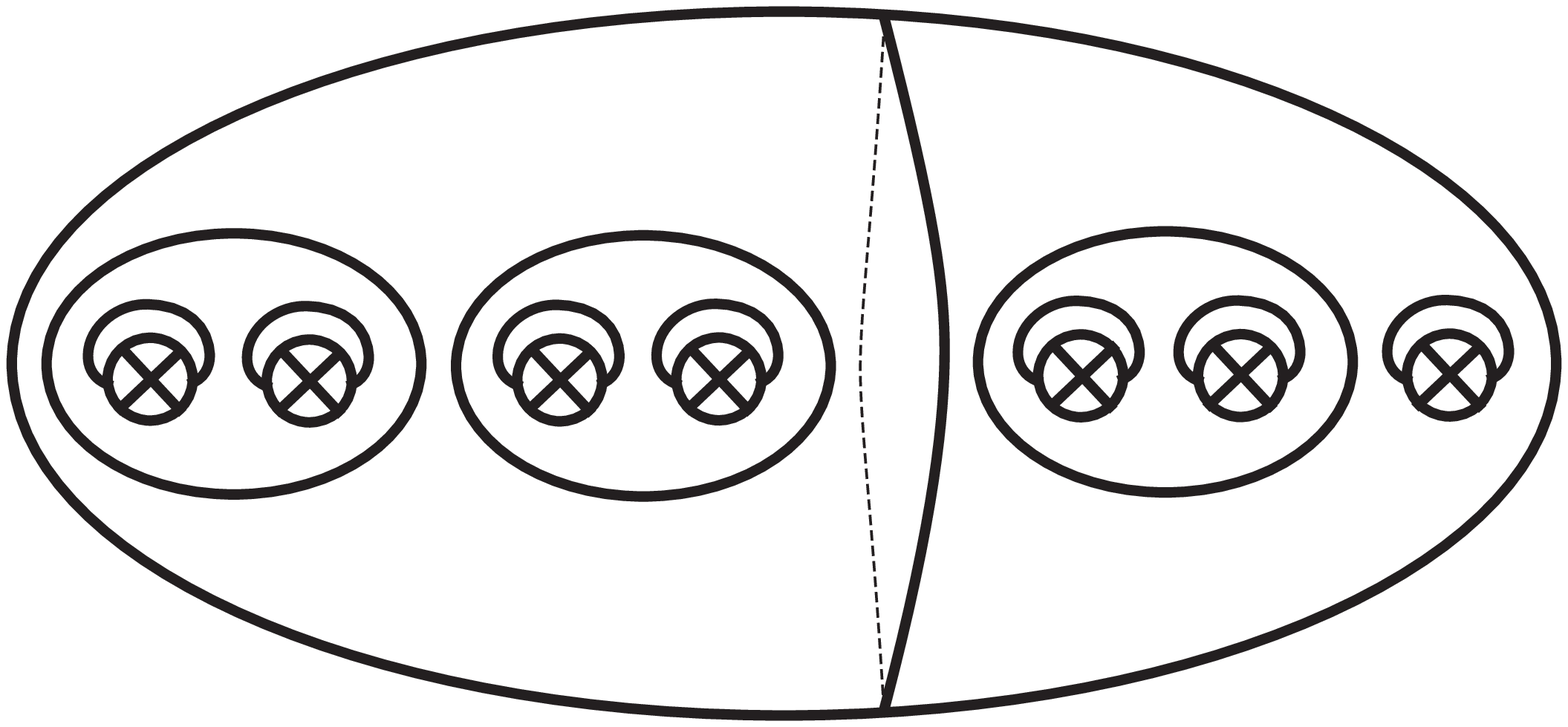} \hspace{0.1in}  \epsfxsize=2in \epsfbox{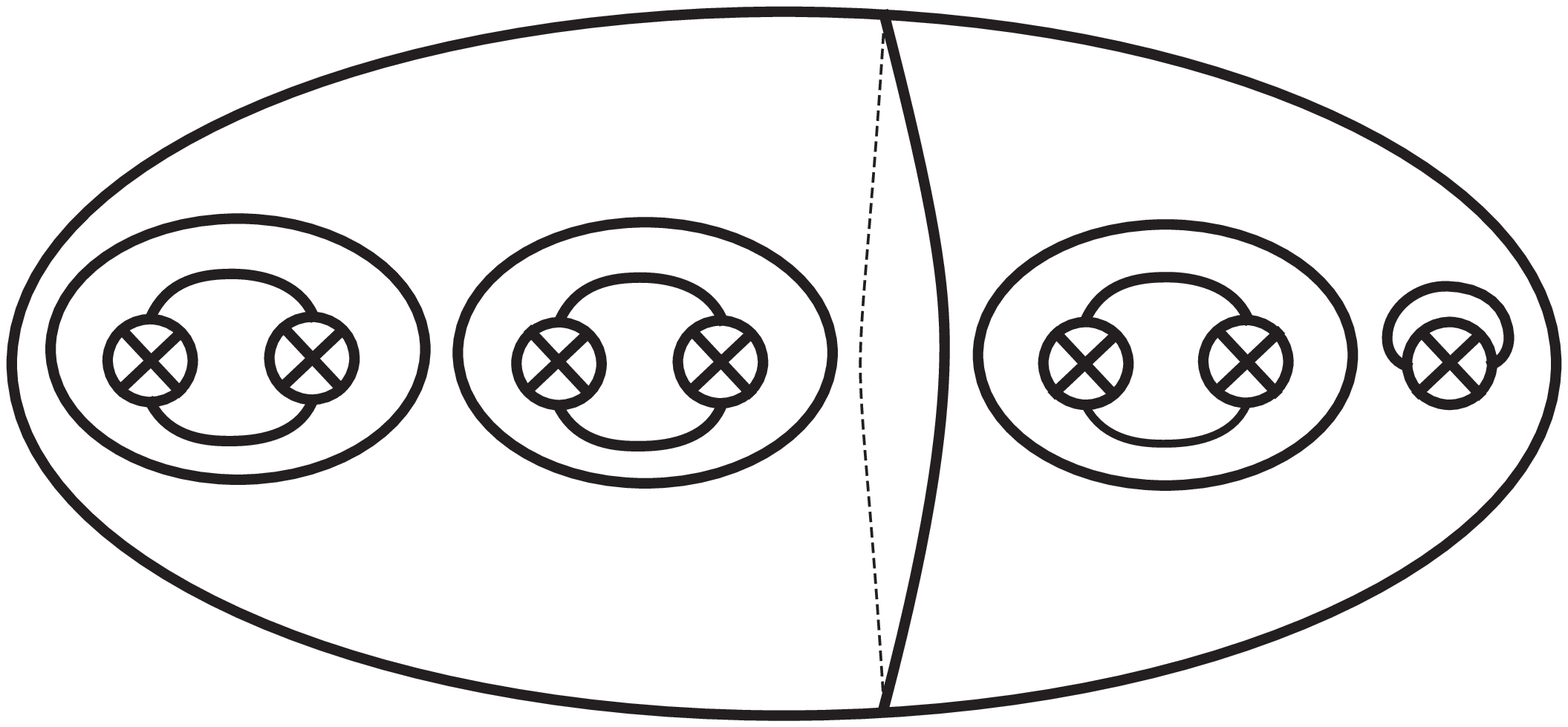}
\caption{Pants decompositions on a surface of genus 7}
\label{Fig0}
\end{center}
\end{figure}

\begin{lemma}
\label{dim} Let $N$ be a nonorientable surface of genus $g \geq 2$ with $n$ boundary components. Suppose that $(g, n) \neq (2, 0)$.
Let $a_r= 3r+n-2$ and $b_r = 4r +n -2$ if $g = 2r+1$, and let $a_r= 3r+n-4$ and $b_r= 4r +n -4$ if $g =2r$.
Then there is a maximal simplex of dimension $q$ in $\mathcal{C}(N)$ if and only if $a_r \leq q \leq b_r$.\end{lemma}

We remind that a simplicial map $\tau : \mathcal{C}(N) \rightarrow \mathcal{C}(N)$ is called superinjective
if it satisfies the following condition: If $[a], [b]$ are two vertices in $\mathcal{C}(N)$, then
$i([a], [b]) = 0$ if and only if $i(\tau([a]), \tau([b])) = 0$.

\begin{lemma}
\label{imply} Let $(g, n)= (3, 0)$ or $g + n \geq 4$. If $\tau: \mathcal{C}(N) \rightarrow \mathcal{C}(N)$ is a superinjective simplicial
map, then $\tau$ satisfies the connectivity property.
\end{lemma}

\begin{proof} Let $\tau: \mathcal{C}(N) \rightarrow \mathcal{C}(N)$ be a superinjective simplicial map. By Lemma 3.2 in \cite{Ir5}, $\tau$ is injective. Let $[a]$, $[b]$ be two vertices in $\mathcal{C}(N)$. We have the following:
$[a]$ and $[b]$ are connected by an edge in $\mathcal{C}(N)$ if and only if ($i([a], [b]) = 0$ and $[a] \neq [b]$)  if and only if ($i(\tau([a]), \tau([b])) = 0$ and $\tau([a]) \neq \tau([b])$) if and only if $\tau([a])$ and $\tau([b])$ are connected by an edge in $\mathcal{C}(N)$.\end{proof}\\

Now we prove some properties of simplicial maps that satisfy the connectivity property:

\begin{lemma}
\label{inj} Let $\lambda: \mathcal{C}(N) \rightarrow \mathcal{C}(N)$ be a simplicial map that satisfies the connectivity property.
If $(g, n)= (3, 0)$ or $g + n \geq 4$, then $\lambda$ is injective.
\end{lemma}

\begin{proof} Let $[a]$ and $[b]$ be two distinct vertices in $\mathcal{C}(N)$. If $[a]$ and $[b]$ are connected by an edge, then $\lambda([a])$ and
$\lambda([b])$ are connected by an edge. So, $\lambda([a]) \neq \lambda([b])$. If $[a]$ and $[b]$ are not connected by an edge, then we can find a simple closed curve $c$ such that $[a]$ and $[c]$ are connected by an edge, and $[b]$ and $[c]$ are not connected by an edge. Then we will have that $\lambda([a])$ and $\lambda([c])$ are connected by an edge, and $\lambda([b])$ and $\lambda([c])$ are not connected by an edge. This gives us that $\lambda([a]) \neq \lambda([b])$. Hence, $\lambda$ is injective.\end{proof}

\begin{lemma}
\label{sup} Let $\lambda: \mathcal{C}(N) \rightarrow \mathcal{C}(N)$ be a simplicial map that satisfies the connectivity property. Let $(g, n)= (3, 0)$ or $g + n \geq 4$. If $a$ and $b$ are two nonisotopic simple closed curves on $N$, then $i([a], [b]) = 0$ if and only if $i(\lambda([a]), \lambda([b])) = 0$.\end{lemma}

\begin{proof} Let $a$ and $b$ be two nonisotopic simple closed curves on $N$. Since $\lambda$ is injective by Lemma \ref{inj},
$\lambda([a]) \neq  \lambda([b])$. Then we have the following: $i([a], [b]) = 0$ if and only if $[a]$ and $[b]$ are connected
by an edge (since $[a] \neq [b]$) if and only if $\lambda([a])$ and $\lambda([b])$ are connected by an edge
if and only if $i(\lambda([a]), \lambda([b])) = 0$ (since $\lambda([a]) \neq  \lambda([b])$).\end{proof}\\

If a maximal simplex has the highest possible dimension in $\mathcal{C}(N)$, we will call
it a top dimensional maximal simplex. Since a simplicial map that satisfies the connectivity property is injective,
it sends top dimensional maximal simplices to top dimensional maximal simplices. By following the proofs of
Lemma 3.3 and Lemma 3.4 given by the author in \cite{Ir5}, we obtain the following results which show that
simplicial maps that satisfy the connectivity property preserve the adjacency and nonadjacency relation w.r.t.
top dimensional maximal simplices in $\mathcal{C}(N)$.

\begin{lemma}
\label{adjacent} Suppose that $(g, n)= (3, 0)$ or $g + n \geq 4$.
Let $\lambda: \mathcal{C}(N) \rightarrow \mathcal{C}(N)$ be a simplicial map that satisfies the connectivity property.
Let $P$ be a pair of pants decomposition on $N$ which corresponds to a top dimensional maximal
simplex in $\mathcal{C}(N)$. Let $a, b \in P$ such that $a$ is adjacent to $b$ w.r.t.
$P$. There exists $a'  \in \lambda([a])$ and $b'  \in \lambda([b])$ such that $a'$ is adjacent
to $b'$ w.r.t. $P'$ where $P'$ is a set of pairwise disjoint curves representing $\lambda([P])$ containing $a', b'$.
\end{lemma}

\begin{lemma}
\label{nonadjacent} Suppose that $g + n \geq 4$. Let $\lambda: \mathcal{C}(N) \rightarrow \mathcal{C}(N)$
be a simplicial map that satisfies the connectivity property. Let $P$ be a pair of pants decomposition on $N$ which
corresponds to a top dimensional maximal simplex in $\mathcal{C}(N)$. Let $a, b \in P$ such that $a$ is not adjacent
to $b$ w.r.t. $P$. There exists $a'  \in \lambda([a])$ and $b'  \in \lambda([b])$ such that $a'$ is not adjacent
to $b'$ w.r.t. $P'$ where $P'$ is a set of pairwise disjoint curves representing $\lambda([P])$ containing $a', b'$.
\end{lemma}

\begin{figure}
\begin{center}
\epsfxsize=1.8in \epsfbox{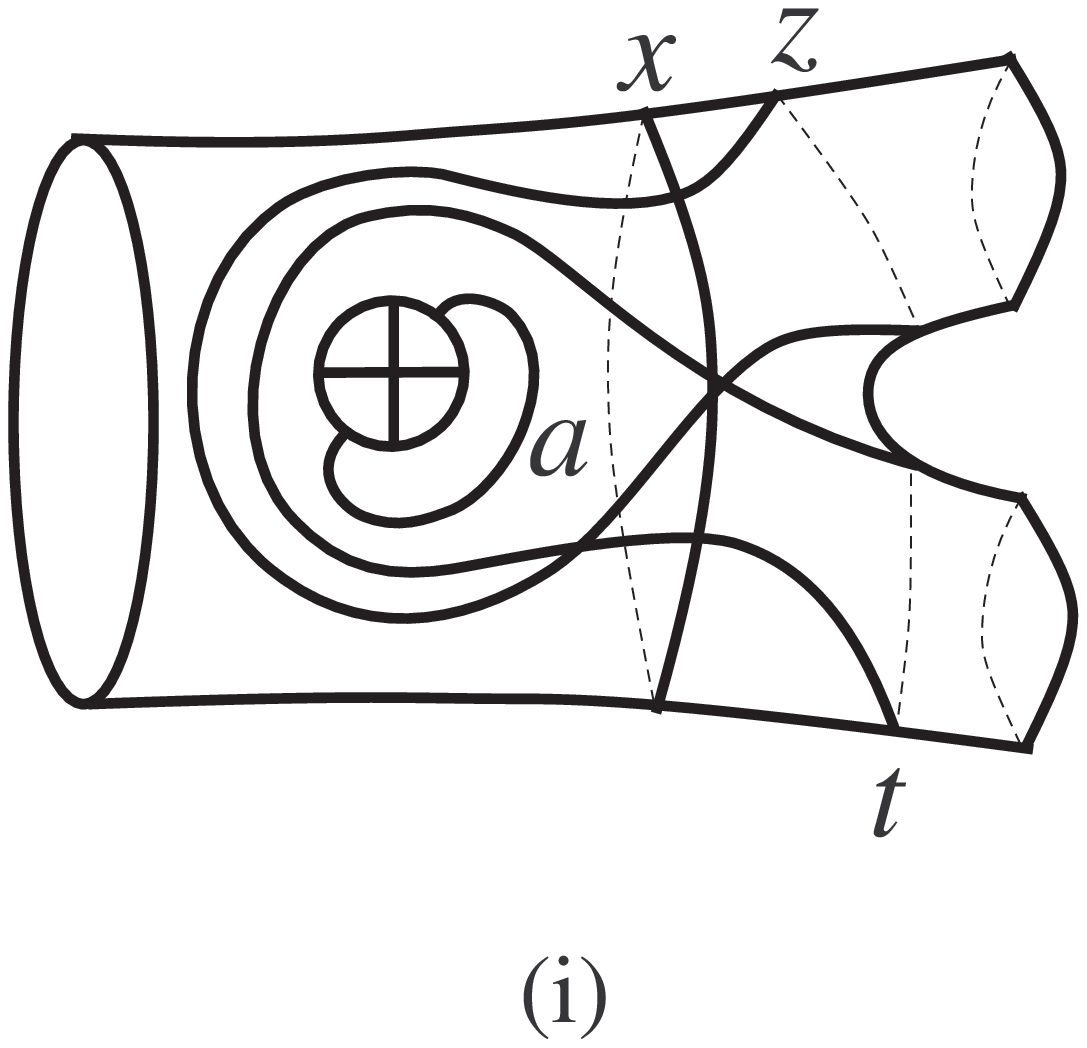} \hspace{0.5in}
\epsfxsize=1.4in \epsfbox{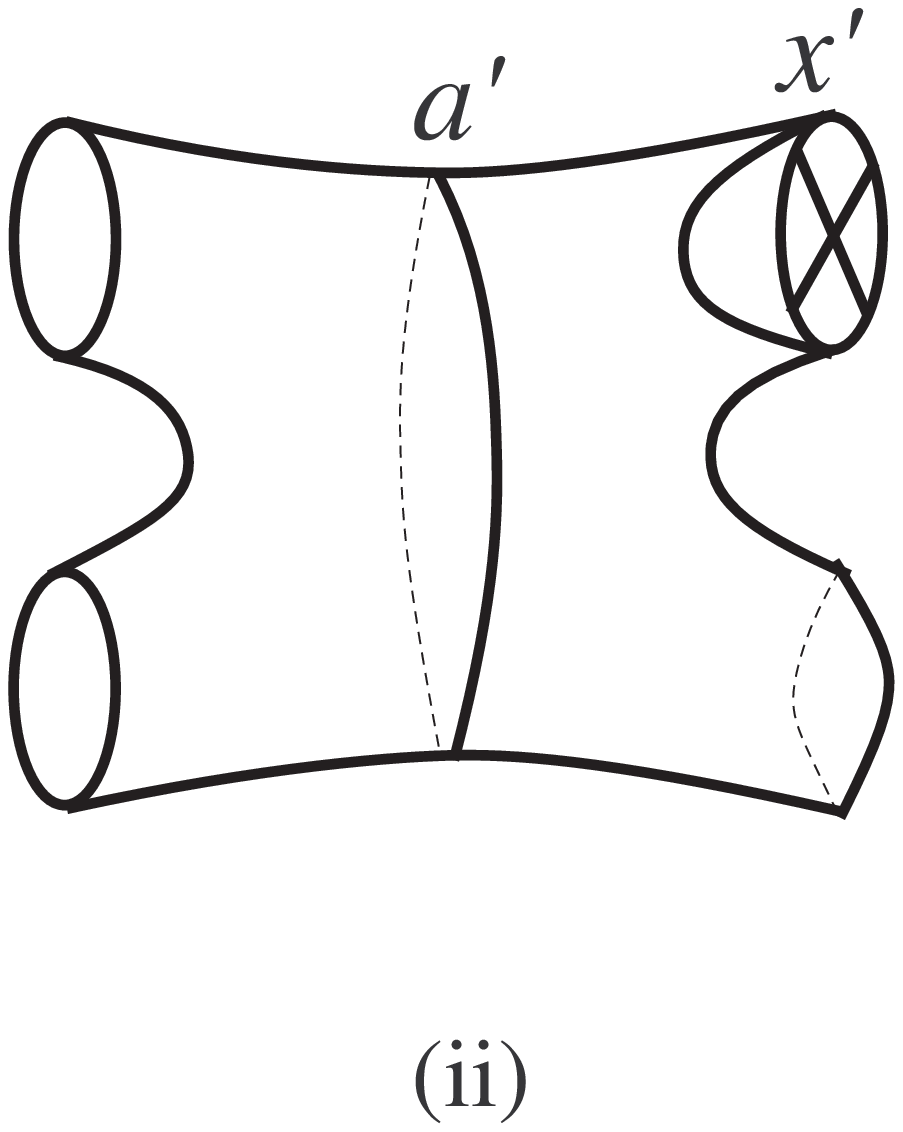}
\caption{Curve configuration I}
\label{Fig1}
\end{center}
\end{figure}

\begin{lemma}
\label{1-sided-co} Suppose that $g=1$ and $n \geq 3$. Let $\lambda: \mathcal{C}(N) \rightarrow \mathcal{C}(N)$ be a simplicial map that satisfies the connectivity property. If $a$ is a 1-sided simple closed curve on $N$, then $\lambda([a])$ is the isotopy class of a 1-sided simple
closed curve on $N$.
\end{lemma}

\begin{proof} Suppose that $g=1$ and $n \geq 3$. Let $a$ be a 1-sided simple closed curve on $N$. Since genus is 1, the complement of
$a$ is orientable. Let $a' \in \lambda([a])$.

Suppose $a'$ is a 2-sided separating simple closed curve. If $n= 3$, then we complete $a$ to a pair of pants decomposition $P=\{a, x\}$ as shown in Figure \ref{Fig1} (i). Let $P'$ be a set of pairwise disjoint
elements of $\lambda([P])$ containing $a'$, and let $x' \in \lambda([x]) \cap P'$. Since $a'$ is separating, $x'$ has to be a 1-sided curve as shown in Figure \ref{Fig1} (ii).
Let $z$ and $t$ be as shown in \ref{Fig1} (i). We see that $a, x, z, t$ are pairwise nonisotopic, each of $z$ and $t$
intersects $x$ essentially (i.e. have nonzero geometric intersection), and each of them is disjoint from $a$. Let $z' \in \lambda([z]), t' \in \lambda([t])$ such that $z', t', a', x'$ intersect minimally. Since $\lambda$ is injective by Lemma \ref{inj}, $a', x', z', t'$ are pairwise nonisotopic.
By Lemma \ref{sup}, we see that each of $z', t'$ is disjoint from $a'$, and intersects $x'$ essentially. So, $z', t'$ both have to lie in the subsurface, $N_1$, a projective plane with two boundary components which contains $x'$ and has $a'$ as a boundary component as shown in
Figure \ref{Fig1} (ii). This gives a contradiction, since in $N_1$ there are no two nontrivial, nonisotopic simple closed curves which intersect 
$x'$ essentially, see Scharlemann's result in \cite{Sc}.

\begin{figure}
\begin{center}
\epsfxsize=3.7in \epsfbox{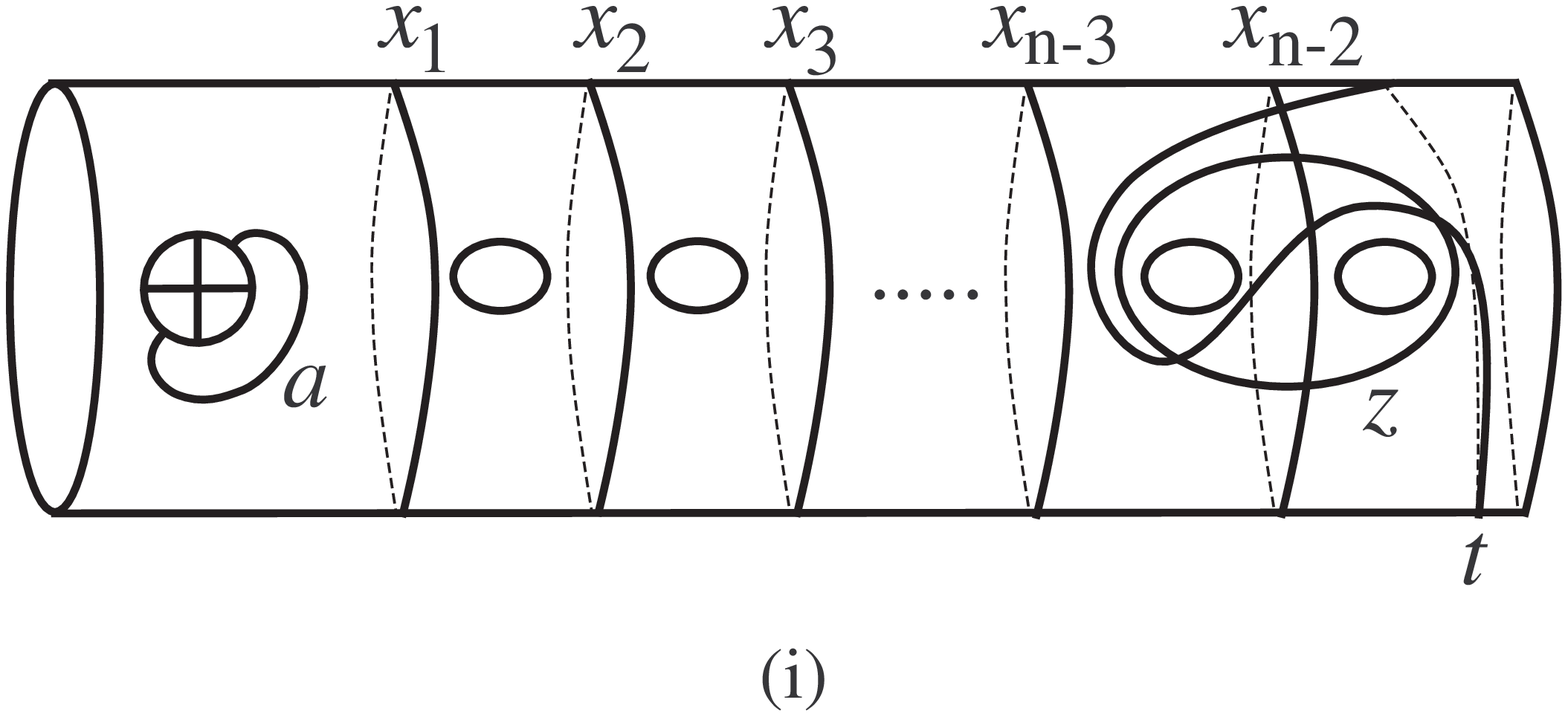}  \vspace{0.1in}
\epsfxsize=3.7in \epsfbox{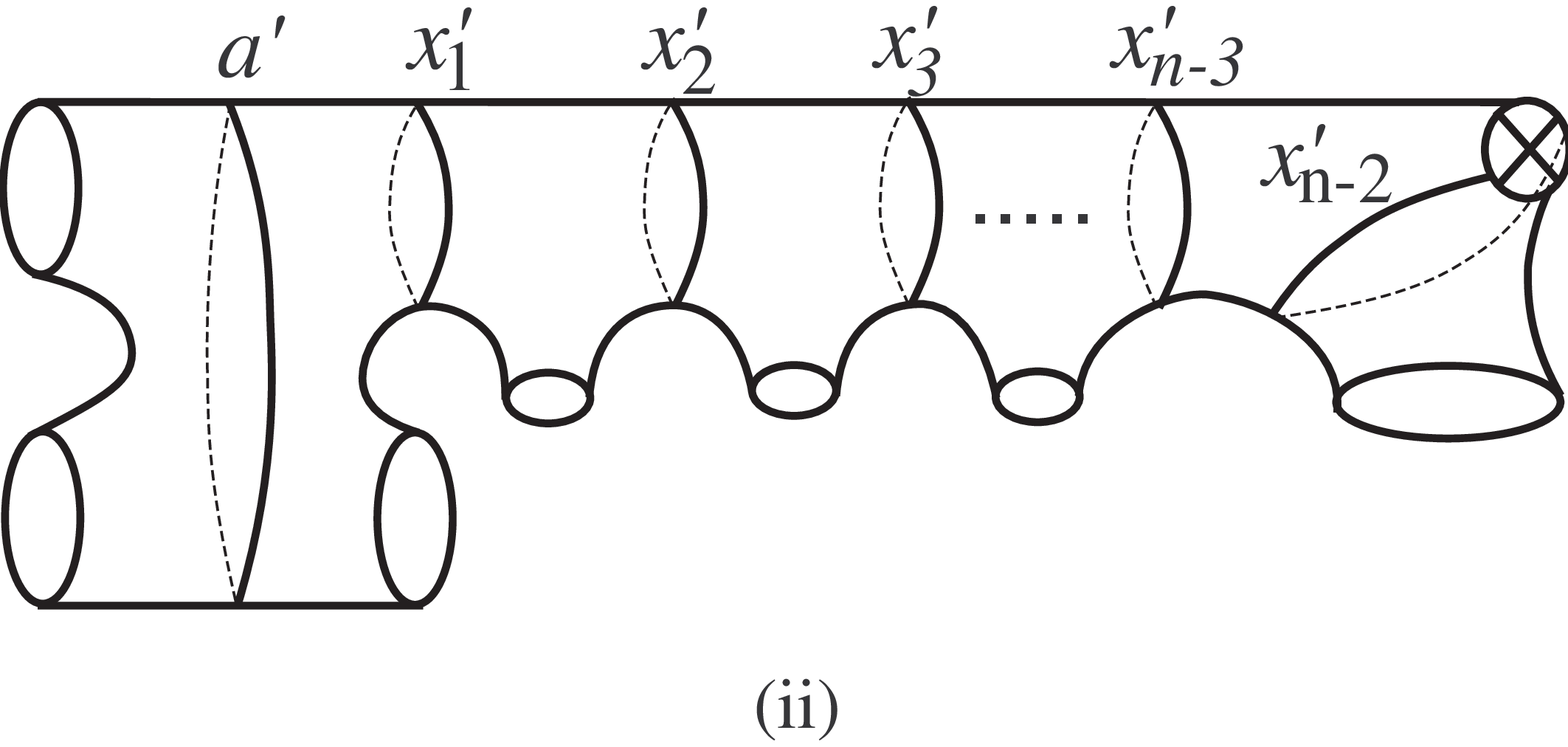}
\caption{Curve configuration II }
\label{Fig2-new}
\end{center}
\end{figure}

If $n \geq 4$, then we complete $a$ to a pair of pants decomposition $P= \{a, x_1, \cdots, x_{n-2}\}$ as shown in
Figure \ref{Fig2-new} (i). $P$ corresponds to a top dimensional maximal simplex. Let $P'$ be a set of pairwise disjoint
curves representing $\lambda([P])$ containing $a'$. Let $x_i' \in \lambda([x_i]) \cap P'$ for each $i= 1, 2, \cdots, n-2$.
By Lemma \ref{adjacent} and Lemma \ref{nonadjacent}, $\lambda$ preserves adjacency and nonadjacency w.r.t. top dimensional
maximal simplices. By using that adjacency and nonadjacency are preserved, $n \geq 4$ and $g=1$, we see that the elements
of $P'$ should be as shown in Figure \ref{Fig2-new} (ii). Consider the curves $z, t$ given in Figure \ref{Fig2-new} (i).
We see that $[x_{n-3}], [x_{n-2}], [z], [t]$ are all distinct
elements and each of $z$ and $t$ intersects $x_{n-2}$ essentially, and each of them is disjoint from $x_{n-3}$. Let
$z' \in \lambda([z]), t' \in \lambda([t])$ such that each of $z'$ and $t'$ intersects each of $x'_{n-3}, x'_{n-2}$ minimally.
Since $\lambda$ is injective by Lemma \ref{inj}, $[x'_{n-3}], [x'_{n-2}], [z'], [t']$ are all distinct elements. By
Lemma \ref{sup}, we see that each of $z', t'$ is disjoint from $x'_{n-3}$, and they intersect $x'_{n-2}$ essentially. So,
$z'$ and $t'$ both have to lie in the subsurface, $N_1$, which is a projective plane with two boundary components containing
$x'_{n-2}$, and having $x'_{n-3}$ as a boundary component as shown in Figure \ref{Fig2-new} (ii). This gives a contradiction as
before. So, $a'$ cannot be a 2-sided separating simple closed curve. Since $g=1$, $a'$ cannot be a 2-sided nonseparating
simple closed curve. Hence, $a'$ is a 1-sided simple closed curve on $N$.\end{proof}

\begin{lemma}
\label{1-sided-cn-10} Suppose that $g=2$ and $n \geq 2$. Let $\lambda: \mathcal{C}(N) \rightarrow \mathcal{C}(N)$
be a simplicial map that satisfies the connectivity property. If $a$ is a 1-sided simple closed curve on $N$ whose
complement is nonorientable, then $\lambda(a)$ is not the isotopy class of a separating simple closed curve on $N$.
\end{lemma}

\begin{proof} Let $a$ be a 1-sided simple closed curve on $N$ whose complement is nonorientable.
Let $a' \in \lambda([a])$. Suppose $a'$ is a separating simple closed curve on $N$.

\begin{figure}
\begin{center}
\epsfxsize=1.7in \epsfbox{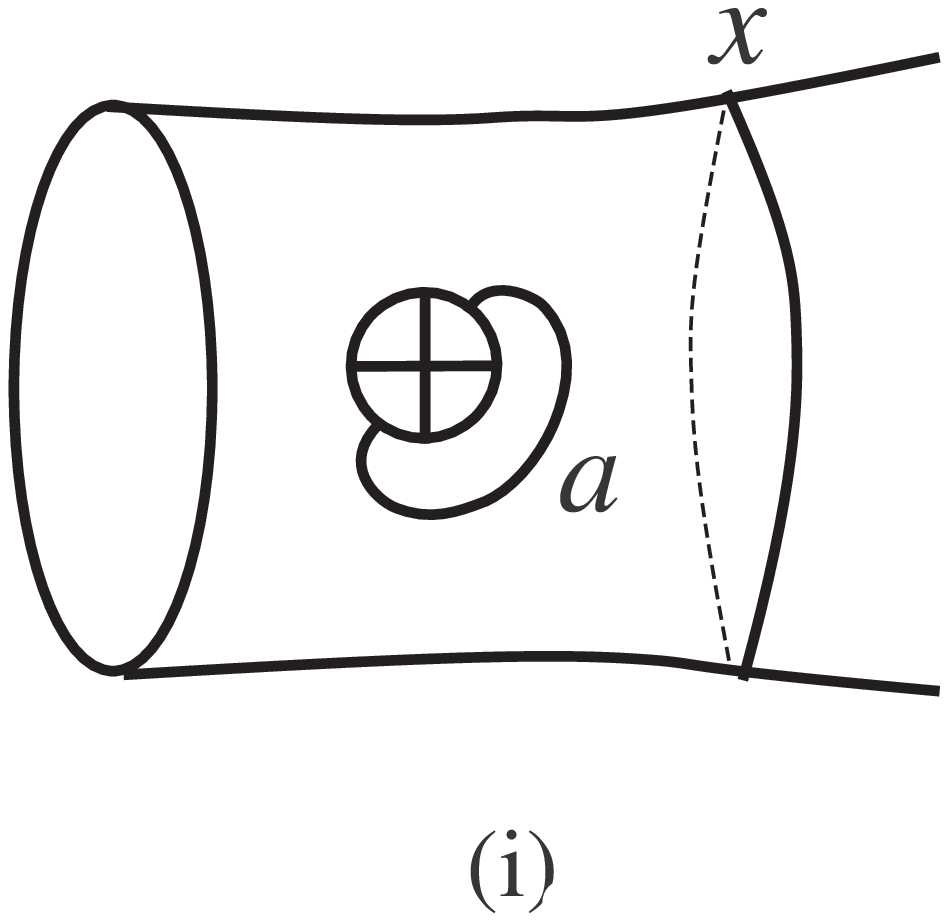} \hspace{0.25in} \epsfxsize=3.5in \epsfbox{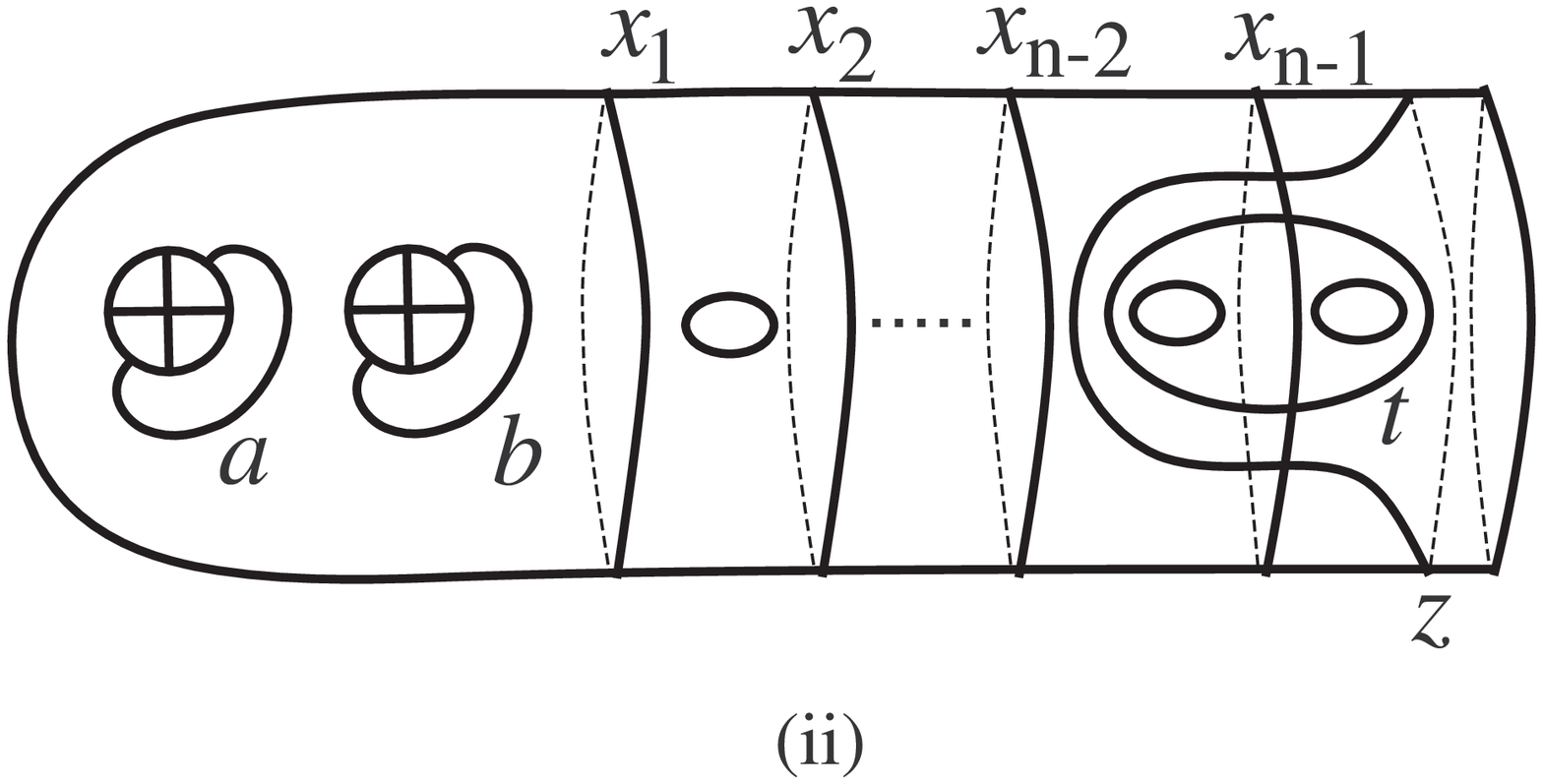}

\epsfxsize=1.8in \epsfbox{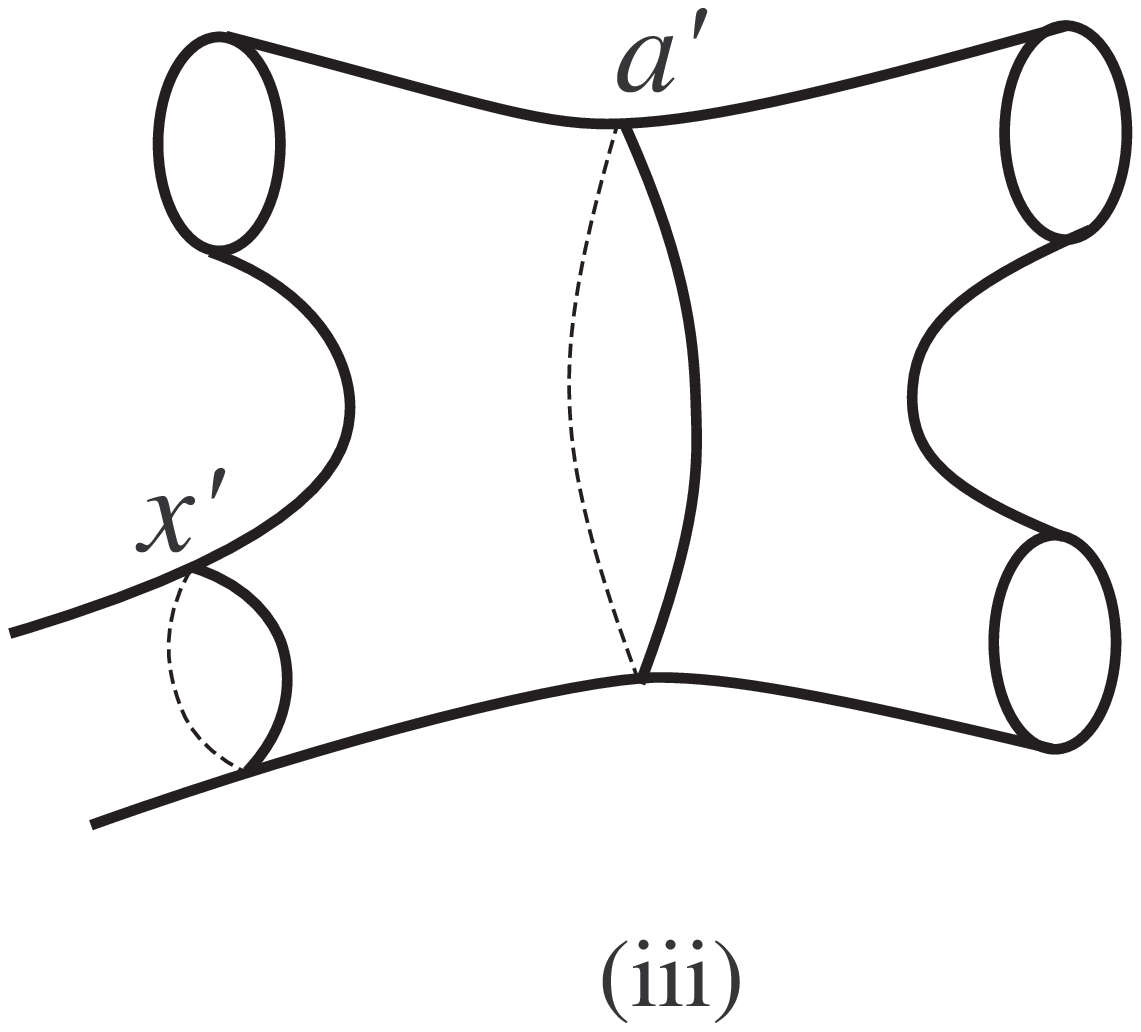} \hspace{0.25in} \epsfxsize=3.7in \epsfbox{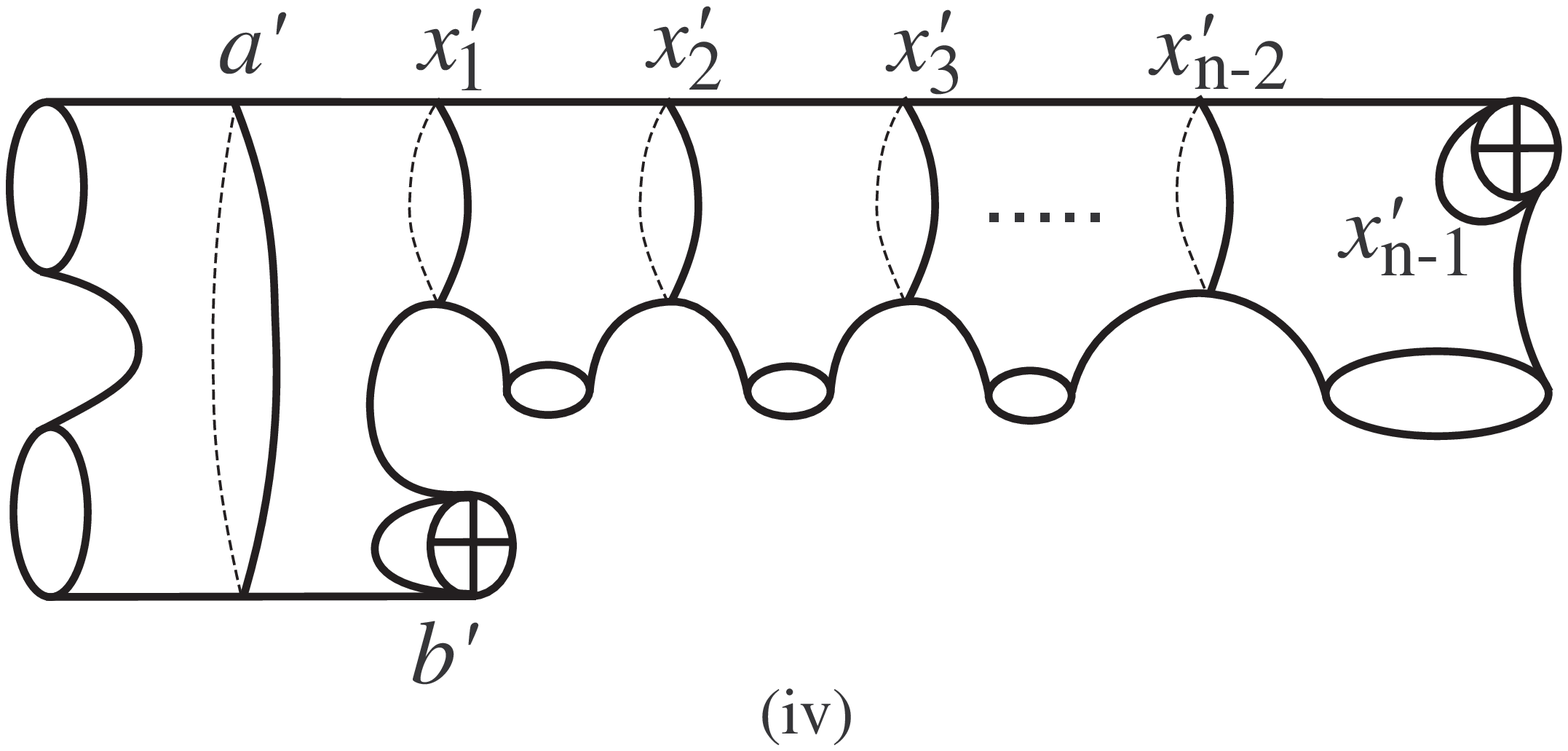}
\caption{Curve configuration III}
\label{Fig3}
\end{center}
\end{figure}

We complete $a$ to a curve configuration $a, x$ as shown in Figure \ref{Fig3} (i), using a boundary component $\partial_1$ of $N$.
Let $T$ be the subsurface of $N$ having $x$ and $\partial_1$ as its boundary as shown in Figure \ref{Fig3} (i). $T$ is a projective
plane with two boundary components. The curves $a$ and $x$ can be completed to a pair of pants decomposition $P$ on $N$ which
corresponds to a top dimensional maximal simplex in $\mathcal{C}(N)$. Let $P'$ be a set of pairwise disjoint curves representing
$\lambda([P])$ containing $a'$, and let $x' \in \lambda([x]) \cap P'$. By Lemma \ref{adjacent} and Lemma \ref{nonadjacent} we know that
$\lambda$ preserves adjacency and nonadjacency w.r.t. top dimensional maximal simplices. Since $a$ is adjacent to only $x$ w.r.t. $P$,
the curve $a'$ is adjacent to only $x'$ w.r.t. $P'$. Since $g=2$ and $n \geq 2$, there is a curve $z$ in $P$ such that $x$ is adjacent
to $z$ w.r.t. $P$ and $a$ is different from $z$. Since $a'$ is a separating curve and $a'$ is only adjacent to $x'$ w.r.t. $P'$, and
$x'$ is adjacent to at least two curves w.r.t. $P'$, we see that there is a four holed sphere, $R_1$, having $x'$ and three boundary
components of $N$ (if exists) as its boundary components such that $a'$ divides $R_1$ into two pair of pants as shown in Figure \ref{Fig3}
(iii). Hence, if $(g, n) = (2, 2)$ we get a contradiction as there are not enough boundary components.

Suppose $g =2, n \geq 3$. We complete $a$ to a pair of pants decomposition $P = \{a, b, x_1, \cdots, x_{n-1}\}$, which corresponds to a top
dimensional maximal simplex, as shown in Figure \ref{Fig3} (ii). In the figure we see that $a$ is adjacent to only $b$ and $x_1$ w.r.t.
$P$. Let $P'$ be a set of pairwise disjoint curves representing $\lambda([P])$ containing $a'$.
Let $b' \in \lambda([b]) \cap P'$, $x_i' \in \lambda([x_i]) \cap P'$ for $i = 1, 2, \cdots, n-1$.
By using that $a'$ is separating and $\lambda$ preserves adjacency and nonadjacency w.r.t. $P$,
we see that $P'$ is as shown in Figure \ref{Fig3} (iv). By choosing
curves $z, t$ as shown in Figure \ref{Fig3} (ii) we get a contradiction as in Lemma \ref{1-sided-co}.
Hence, $a'$ cannot be separating simple closed curve.\end{proof}

\begin{figure}
\begin{center}
\epsfxsize=2.7in \epsfbox{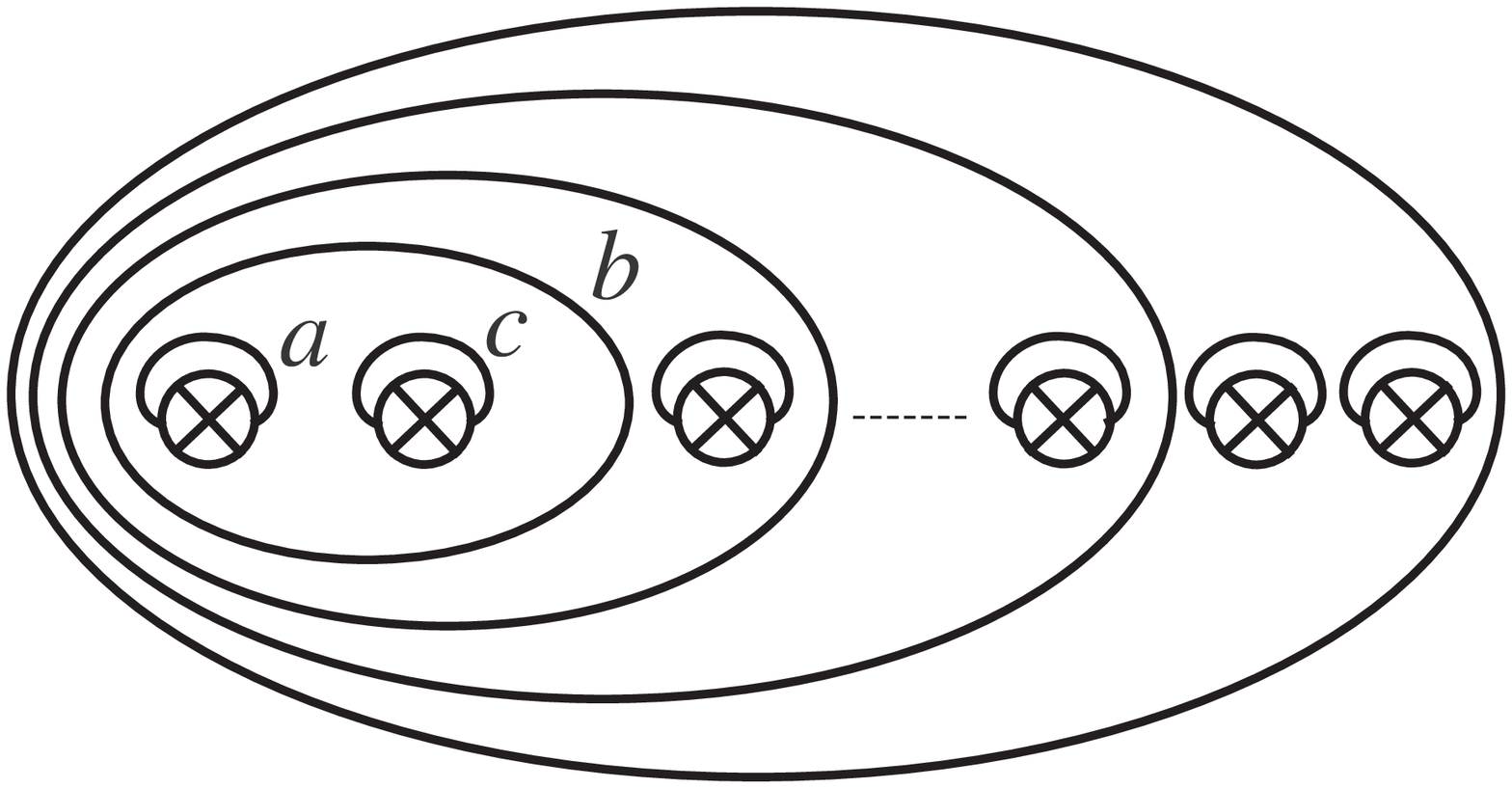} \hspace{0.1in}
\epsfxsize=2.7 in \epsfbox{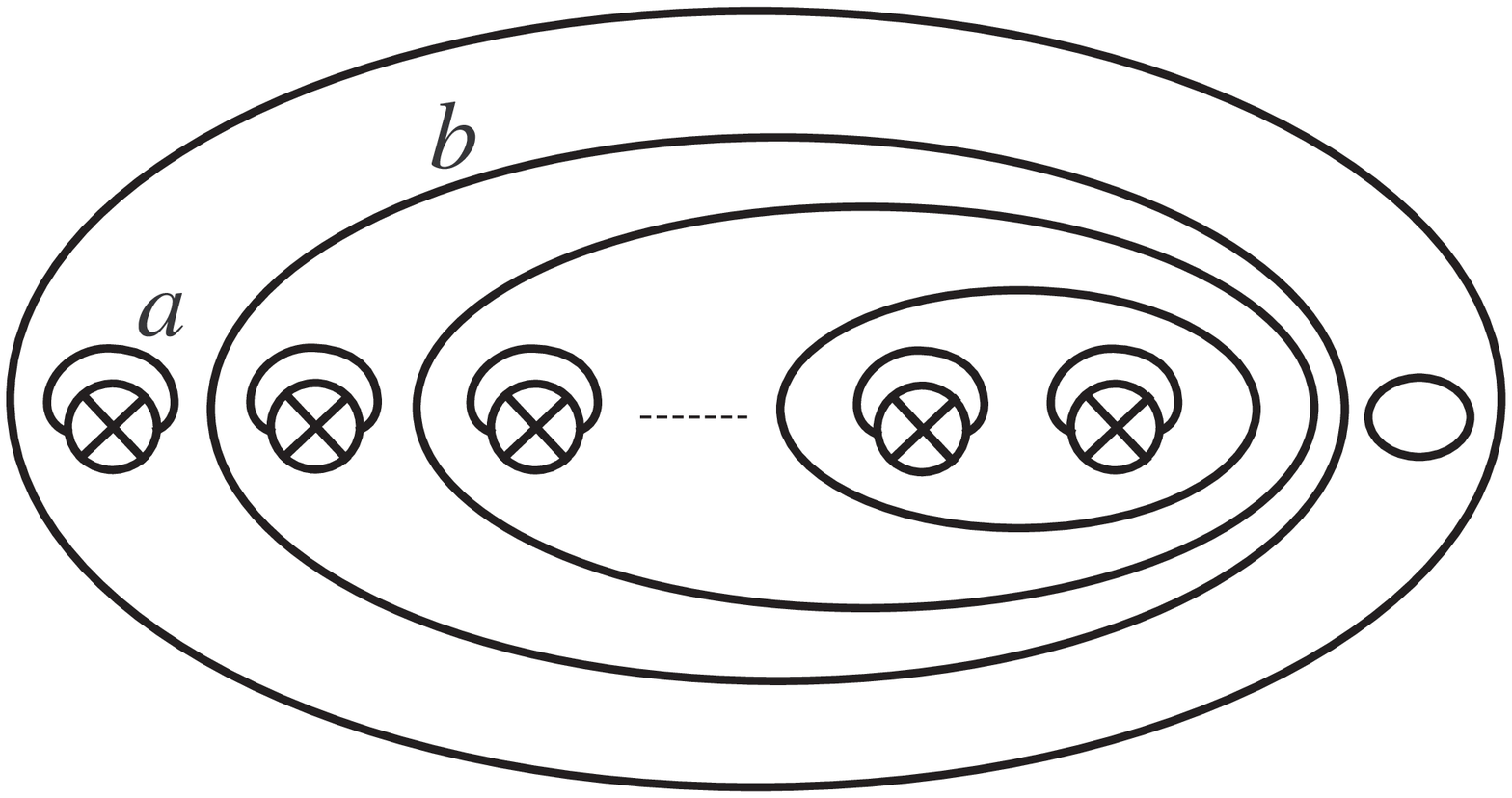}

\small{ \hspace{-0.1in} (i)  \hspace{2.4in} (ii)}

\epsfxsize=2.7in \epsfbox{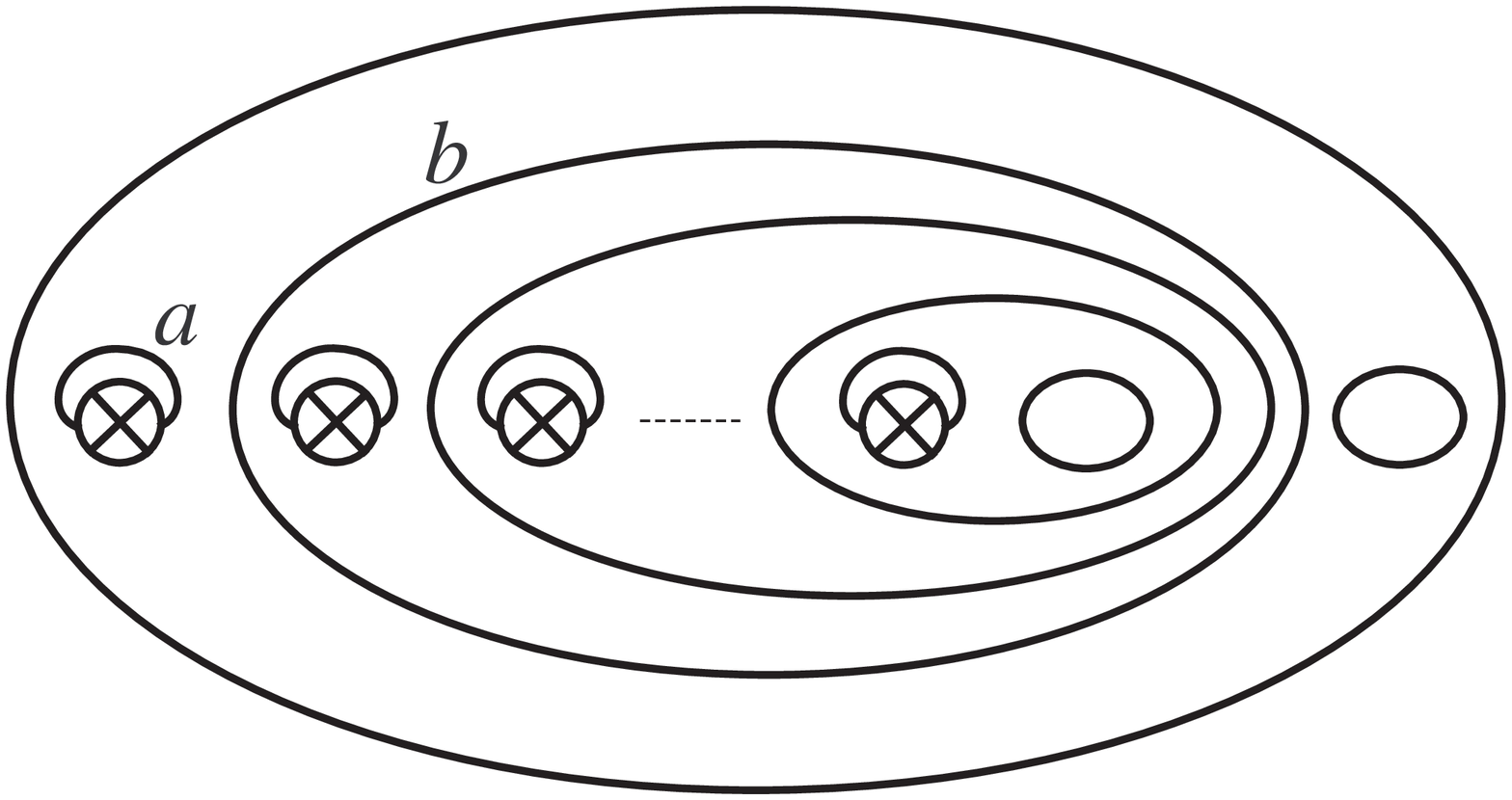}

\small{ \hspace{-0.1in} (iii)}

\caption{Pants decompositions}
\label{Fig4}
\end{center}
\end{figure}

\begin{lemma}
\label{1-sided-cn-11} Suppose that $g \geq 3$. Let $\lambda: \mathcal{C}(N) \rightarrow \mathcal{C}(N)$ be a simplicial map
that satisfies the connectivity property. If $a$ is a 1-sided simple closed curve on $N$ whose complement is nonorientable,
then $\lambda(a)$ is not the isotopy class of a separating simple closed curve on $N$.
\end{lemma}

\begin{proof} Let $a$ be a 1-sided simple closed curve on $N$ whose complement is nonorientable.
Let $a' \in \lambda([a])$. If $(g, n)=(3, 0)$ then there is
no nontrivial separating simple closed curve on $N$. So, the result follows. Assume that $a'$ is a separating simple
closed curve on $N$.

(a) If $n=0$ and $g \geq 4$, then we complete $a$ to a pants decomposition
$P$ such that $P$ corresponds to a top dimensional maximal simplex in $\mathcal{C}(N)$, and $a$ is adjacent to only two curves
w.r.t. $P$ (see Figure \ref{Fig4} (i)). Let $P'$ be a set of pairwise disjoint curves representing $\lambda([P])$ containing $a'$.
Since $\lambda$ is injective by Lemma \ref{inj}, $P'$ also corresponds to a top dimensional maximal simplex. Since adjacency and nonadjacency 
w.r.t. top dimensional maximal simplices are preserved by Lemma \ref{adjacent} and Lemma \ref{nonadjacent}, we see that $a'$ has to be 
adjacent to only two simple closed curves w.r.t. $P'$. This is impossible since $a'$ is a separating curve, $n=0$ and $g \geq 4$. 
So, we get a contradiction.

\begin{figure}
\begin{center}
\epsfxsize=3.8in \epsfbox{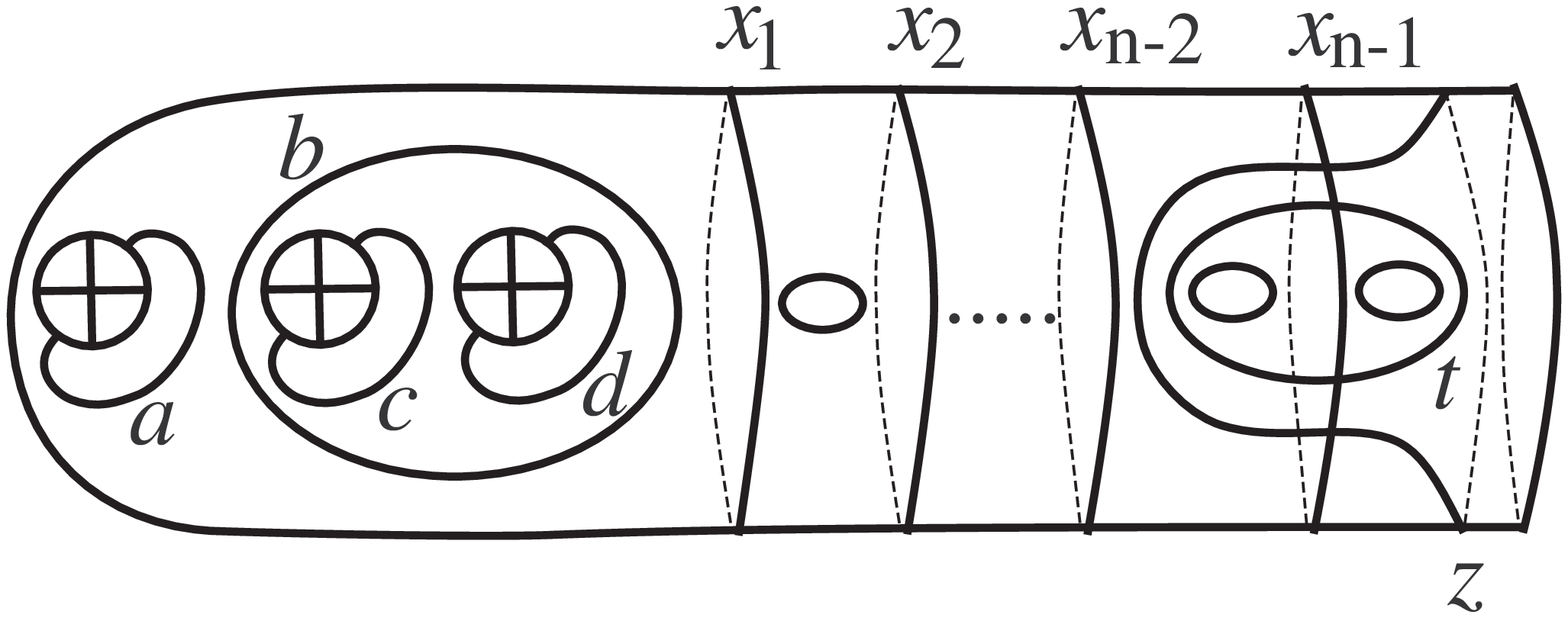}

\vspace{-0.4cm} (i) \vspace{0.5cm}

\epsfxsize=3.8in \epsfbox{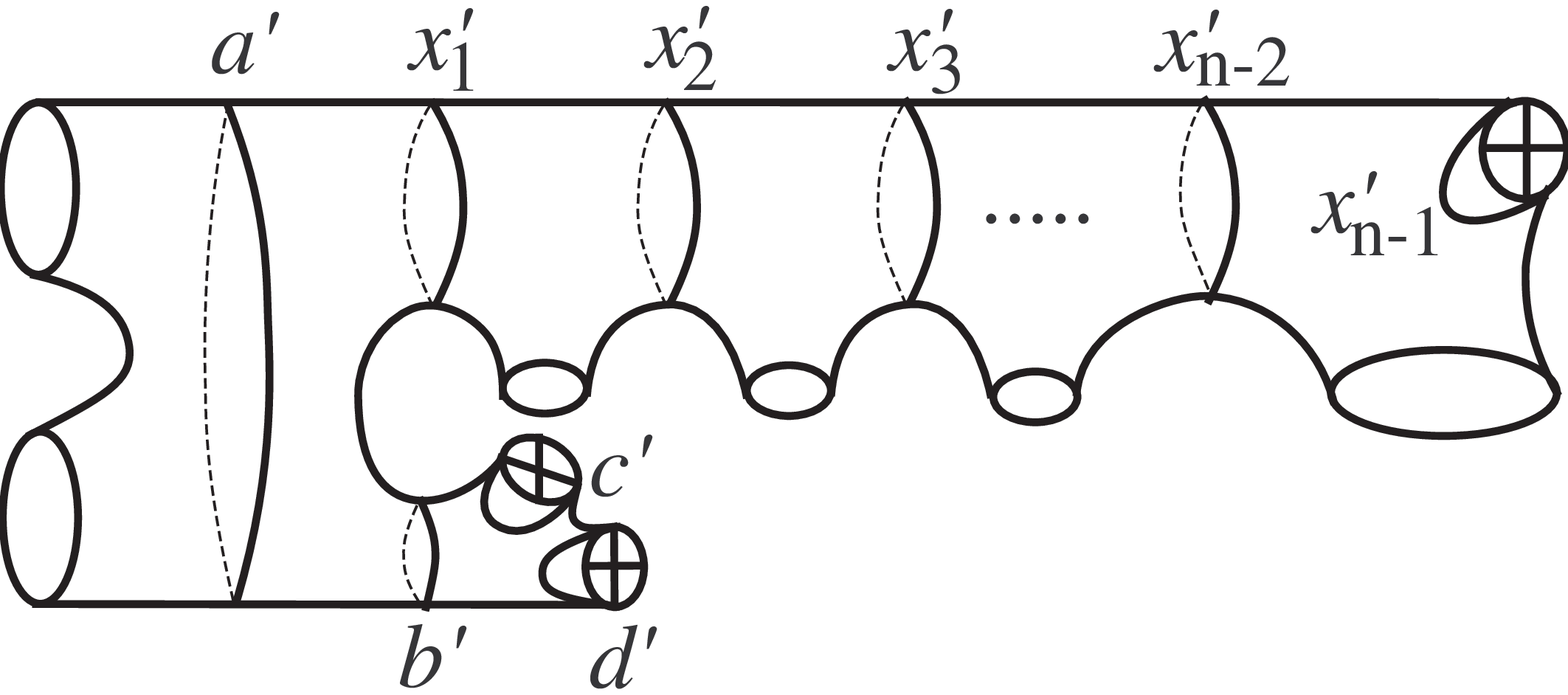}

\vspace{-0.4cm} (ii) \vspace{0.5cm}

\epsfxsize=4.6in \epsfbox{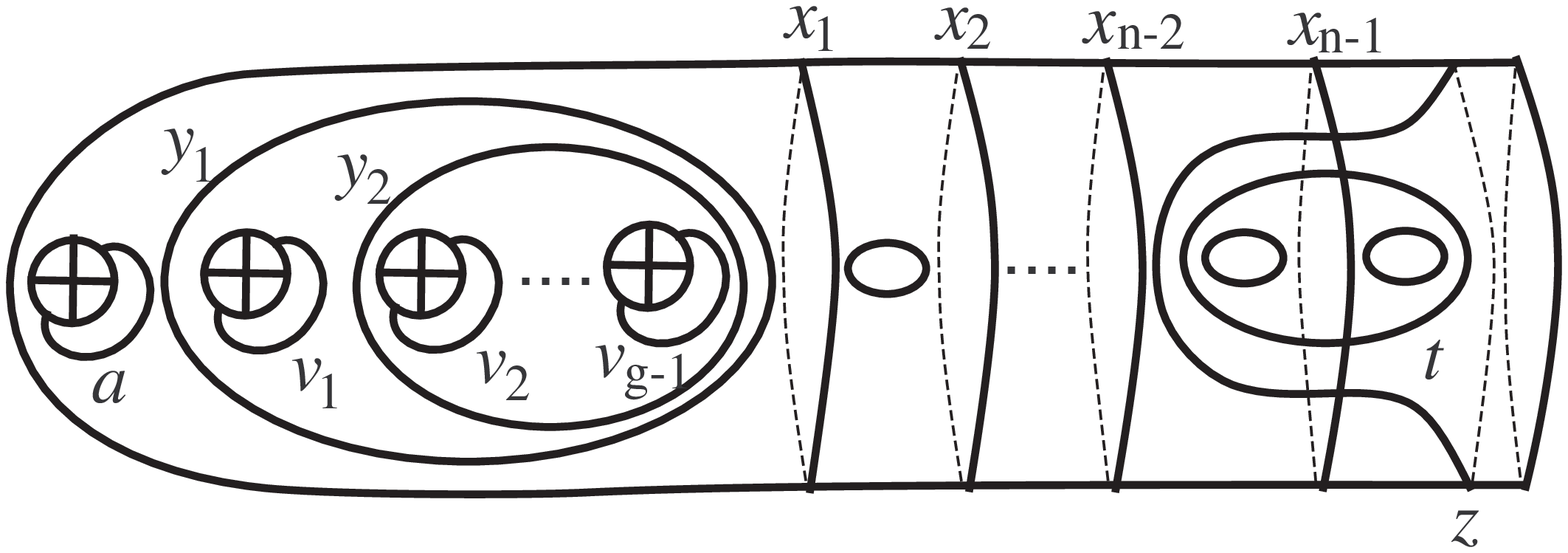}

\vspace{-0.4cm} (iii) \vspace{0.5cm}

\hspace{0.8cm} \epsfxsize=4.7in \epsfbox{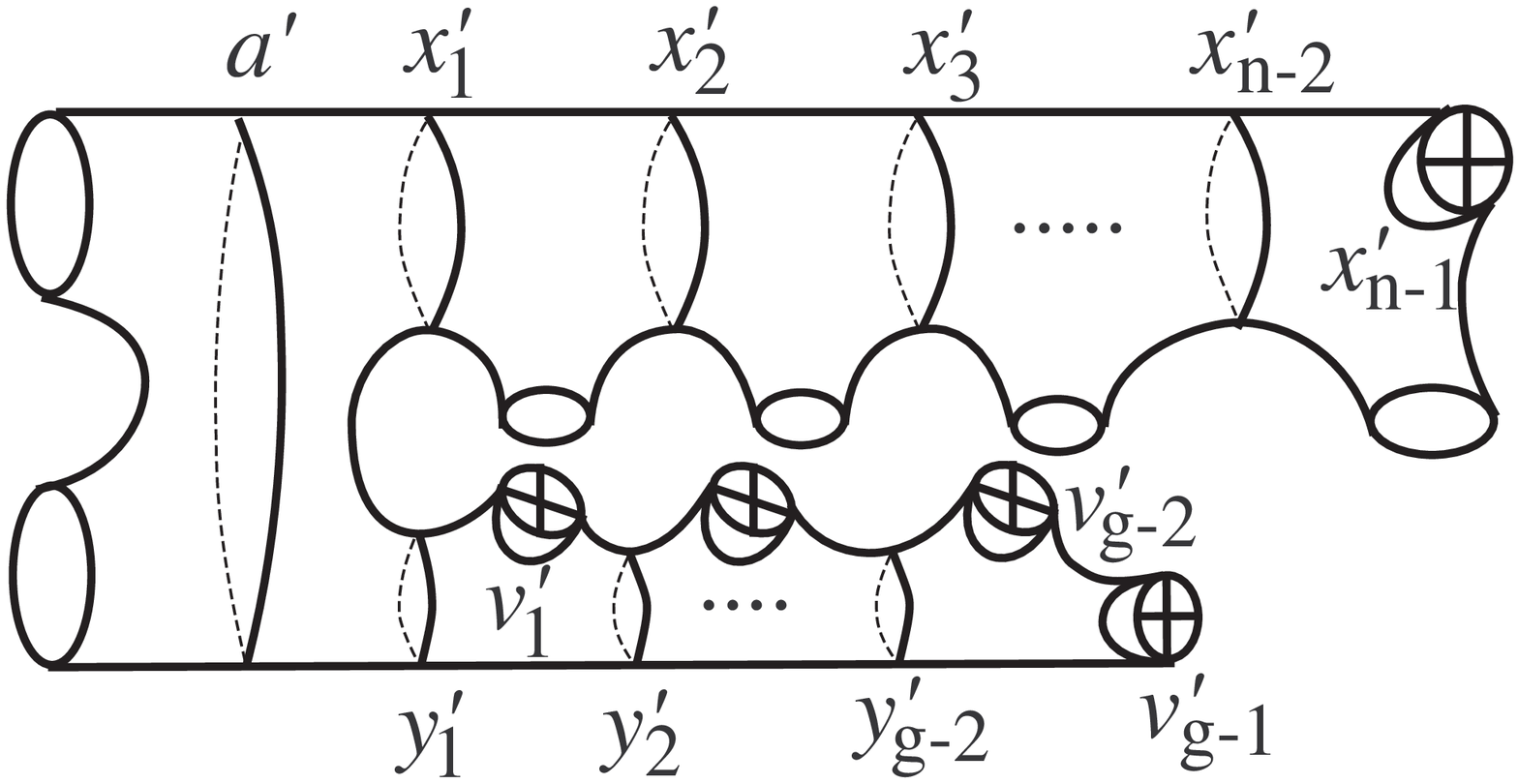}

\vspace{-0.1cm} (iv)  \vspace{0.5cm}

\caption{Curve configuration IV}
\label{Fig10}
\end{center}
\end{figure}

(b) If $n=1$ or $n = 2$, and $g \geq 3$, then we complete $a$ to a pair of pants
decomposition $P$ such that $P$ corresponds to a top dimensional maximal simplex in $\mathcal{C}(N)$, and $a$ is adjacent to only
one curve w.r.t. $P$ as shown in Figure \ref{Fig4} (ii) and Figure \ref{Fig4} (iii). Since adjacency and nonadjacency w.r.t. top
dimensional maximal simplices are preserved by Lemma \ref{adjacent} and Lemma \ref{nonadjacent}, we see that $a'$ has to be adjacent
to only one simple closed curve in the image pants decomposition which corresponds to a top dimensional maximal simplex. This is
impossible since $a'$ is separating, $g \geq 3$ and there is only one or two boundary components. So, we get a contradiction.

(c) If $n \geq 3$ and $g \geq 3$, then we get a contradiction as follows: If $g=3$, then we complete $a$ to a
pair of pants decomposition $P= \{a, b, c, d, x_1, \cdots, x_{n-1}\}$ as shown in Figure \ref{Fig10} (i). We see that $P$
corresponds to a top dimensional maximal simplex in $\mathcal{C}(N)$. Let $P'$ be a set of pairwise disjoint curves representing
$\lambda([P])$ containing $a'$. The set $P'$ also corresponds to a top dimensional maximal simplex in $\mathcal{C}(N)$. Let
$b' \in \lambda([b]) \cap P'$, $c' \in \lambda([c]) \cap P'$, $d' \in \lambda([d]) \cap P'$, and
$x_i' \in \lambda([x_i]) \cap P'$ for $i= 1, \cdots, n-1$. By using that adjacency and nonadjacency are preserved w.r.t. 
top dimensional maximal simplices, we get $P'= \{a', b', c', d', x'_1, \cdots, x'_{n-1}\}$ as shown in Figure \ref{Fig10} (ii). 
Let $z$ and $t$ be as shown in Figure \ref{Fig10} (i). By using the curves $z, t$, we get a contradiction as in the proof of 
Lemma \ref{1-sided-cn-10}. Similarly, if $g \geq 4$, we complete $a$ to a pair of pants decomposition $P$ which corresponds 
to a top dimensional maximal simplex as shown in Figure \ref{Fig10} (iii), and we get the corresponding pair of pants decomposition 
as shown in Figure \ref{Fig10} (iv). Using the curves $z, t$ shown in Figure \ref{Fig10} (iii), we get a contradiction as in the 
proof of Lemma \ref{1-sided-cn-10}.

Hence, $a'$ cannot be separating simple closed curve.\end{proof}

\begin{lemma}
\label{1-sided-cn-2} Let $g \geq 2$. Suppose that $(g, n) = (3, 0)$ or $g+n \geq 4$.
Let $\lambda: \mathcal{C}(N) \rightarrow \mathcal{C}(N)$ be a simplicial map that satisfies the connectivity property.
If $a$ is a 1-sided simple closed curve on $N$ whose complement is nonorientable, then $\lambda(a)$ is the isotopy
class of a 1-sided simple closed curve whose complement is nonorientable.
\end{lemma}

\begin{proof} Let $a$ be a 1-sided simple closed curve whose complement is nonorientable.
Let $a' \in \lambda([a])$. By Lemma \ref{1-sided-cn-10} and Lemma \ref{1-sided-cn-11}, $a'$ cannot be a separating
simple closed curve.

Since $a$ is a 1-sided simple closed curve whose complement is nonorientable, $[a]$ is the vertex of a top dimensional maximal
simplex in $\mathcal{C}(N)$. Since $\lambda$ is injective by Lemma \ref{inj}, the image of this simplex is a top dimensional maximal
simplex containing $[a']$. If $a'$ is a 1-sided simple closed curve whose complement is orientable or a 2-sided nonseparating simple
closed curve, then $[a']$ can't be a vertex of a top dimensional maximal simplex in $\mathcal{C}(N)$. Hence, $a'$ is a 1-sided simple 
closed curve whose complement is nonorientable.\end{proof}\\

The following theorem is given by the author in \cite{Ir5}.

\begin{theorem}
\label{Ir} Let $N$ be a compact, connected, nonorientable surface of genus $g$ with $n$ boundary components. Suppose
that either $(g, n) \in \{(1, 0), (1, 1), (2, 0)$, $(2, 1), (3, 0)\}$ or $g + n \geq 5$. If
$\tau : \mathcal{C}(N) \rightarrow \mathcal{C}(N)$ is a superinjective simplicial map, then $\tau$ is
induced by a homeomorphism $h : N \rightarrow N$.\end{theorem}

The main result:

\begin{theorem} Let $N$ be a compact, connected, nonorientable surface of genus $g$ with
$n$ boundary components. Suppose that either $(g, n) \in \{(1, 0), (1, 1), (2, 0)$, $(2, 1), (3, 0)\}$ or $g + n \geq 5$.
If $\lambda: \mathcal{C}(N) \rightarrow \mathcal{C}(N)$ is a simplicial map that satisfies the connectivity property,
then $\lambda$ is induced by a homeomorphism $h : N \rightarrow N$.\end{theorem}

\begin{proof} The proof follows as in the proof of Theorem \ref{Ir} given by the author in \cite{Ir5}, by using Lemma \ref{sup},
Lemma \ref{1-sided-co} and Lemma \ref{1-sided-cn-2}.\end{proof}\\\\

{\bf Acknowledgments}
\vspace{0.3cm}

We thank Mustafa Korkmaz for some discussions.

eirmak@bgsu.edu

Bowling Green State University

Department of Mathematics and Statistics

Bowling Green, 43403, OH, USA
\end{document}